\documentclass[11pt]{article}
\usepackage{amsmath}
\usepackage{amsthm}
\usepackage{amssymb}
\usepackage{authblk}

\newtheorem{thm}{Theorem}

\newtheorem{cor}{Corollary}

\newtheorem{lemma}{Lemma}

\begin{document}

\vspace*{30px}

\begin{center}
{\Large 
\textbf{Weighted Lattice Paths Enumeration by Gaussian Polynomials} \bigskip}

{\large 
Ivica Martinjak\\[0pt]
Faculty of Science, University of Zagreb\\[0pt]
Bijeni\v cka cesta 32, HR-10000 Zagreb, Croatia\\[0pt]
and \\[0pt]
Ivana Zubac\\[0pt]
Faculty of Mechanical Engineering and Computing, University of Mostar\\[0pt]
Matice hrvatske bb, BA-88000 Mostar, Bosnia and Herzegovina }

{\large 
}

{\large 
}
\end{center}

\begin{abstract}
The Gaussian polynomial in variable $q$ is defined as the $q$-analog of the
binomial coefficient. In addition to remarkable implications of these
polynomials to abstract algebra, matrix theory and quantum computing, there
is also a combinatorial interpretation through weighted lattice paths. This
interpretation is equivalent to weighted board tilings, which can be used to
establish Gaussian polynomial identities. In particular, we prove duals of
such identities and evaluate related sums.
\end{abstract}

\noindent \textbf{Keywords:} Gaussian binomial coefficient, $q$-binomial
coefficient, Gaussian polynomial identities, $q$-analog, lattice paths,
domino tiling, double counting\newline
\noindent \textbf{AMS Mathematical Subject Classifications:} 05A30, 11B65


\section{Introduction}

We let $n$ and $k$ denote non-negative integers. The \emph{Gaussian
polynomial} or \emph{q-binomial coefficient} is defined to be 
\begin{eqnarray*}
{\ n\brack k}_q = \frac{(q^n - 1) (q^{n-1}-1) \cdots (q^{n-k+1} -1 ) }{%
(q^{k} -1) (q^{k-1} -1) \cdots (q -1) }.
\end{eqnarray*}
Thus, we replace each factor $m$ in the formula for binomial coefficient, by 
$q^m-1$ to get Gaussian coefficient. It is easily seen that Gaussian
binomial coefficient reduces to binomial coefficient when $q=1$. More
precisely, by the L'H\^opital's Rule, we have 
\begin{eqnarray*}
\lim_{q \to 1}{\ n\brack k}_q = \binom{n}{k}.
\end{eqnarray*}
More on this subject one can find in the book by P. Cameron \cite{Cameron}.
Recall a classical implication of Gaussian polynomial. The number of
elements in a finite field is a prime power and up to isomorphism there is a
unique field of any prime power order (E. Galois). We let GF($q$) denote the
field with $q$ elements. Then the number of $k$-dimensional subspaces of an $%
n$-dimensional vector space over GF($q$) is ${\ n \brack k}_q$. A proof of
this one can find in a nice overview by H. Cohn \cite{Cohn}. It is also
worth mentioning that Gaussian polynomial are related to \emph{quantum
calculus} and more on this one can find in a book by V. Kac and P. Cheung 
\cite{Kac}.

Having the $q$-analog of the factorial defined as 
\begin{eqnarray*}
[n]_q! = (1+q) (1+q+q^2)\cdots (1+q+\cdots + q^{n-1})
\end{eqnarray*}
we also have 
\begin{eqnarray*}
{\ n\brack k}_q = \frac{[n]_q!}{[k]_q! [n-k]_q!}.
\end{eqnarray*}

Recently, V. Guo and D. Yang \cite{GuoYang} found Gaussian polynomial
identities 
\begin{eqnarray*}
\sum_{k=0}^{\lfloor n/2\rfloor }{\ m+k\brack k}_{q^{2}}{\ m+1\brack n-2k}%
_{q}q^{\binom{n-2k}{2}} &=&{\ m+n\brack n}_{q}, \\
\sum_{k=0}^{\lfloor n/4\rfloor }{\ m+k\brack k}_{q^{4}}{\ m+1\brack n-4k}%
_{q}q^{\binom{n-4k}{2}} &=&\sum_{k=0}^{\lfloor n/2\rfloor }(-1)^{k}{\ m+k%
\brack k}_{q^{2}}{\ m+n-2k\brack n-2k}_{q}
\end{eqnarray*}%
which are $q$-analogues of binomial coefficient identities of Y. Sun, 
\begin{eqnarray*}
\sum_{k=0}^{\lfloor n/2\rfloor }\binom{m+k}{k}\binom{m+1}{n-2k} &=&\binom{m+n%
}{n}, \\
\sum_{k=0}^{\lfloor n/4\rfloor }\binom{m+k}{k}\binom{m+1}{n-4k}
&=&\sum_{k=0}^{\lfloor n/2\rfloor }(-1)^{k}\binom{m+k}{k}\binom{m+n-2k}{m},
\end{eqnarray*}%
respectively. For an overview of binomial coefficients we refer the reader
to classical reference by R. Graham, D. Knuth and O. Patashnik \cite{knuth}.
Being motivated by these results, in this work we aim at finding further
Gaussian polynomial identities.

\section{$q$-binomial coefficients and weighted lattice paths}

As the basic recurrences for Gaussian polynomials we have 
\begin{eqnarray}
{\ n \brack k}_q = {\ n-1 \brack k-1}_q + q^k {\ n-1 \brack k}_q
\label{thefirstrecurrence}
\end{eqnarray}
and 
\begin{eqnarray}
{\ n \brack k}_q =q^{n-k} {\ n-1 \brack k-1}_q + {\ n-1 \brack k}_q
\label{thesecondrecurrence}
\end{eqnarray}
where $0 <k <n$ and 
\begin{eqnarray*}
{\ n \brack 0}_q = {\ n \brack n}_q =1.
\end{eqnarray*}
This follows immediately from the definition. To prove (\ref%
{thefirstrecurrence}) we have 
\begin{eqnarray*}
{\ n \brack k}_q - {\ n-1 \brack k-1}_q &=& \bigg( \frac{q^n-1}{q^k -1} -1 %
\bigg) {\ n-1 \brack k-1}_q \\
&=& q^k \bigg( \frac{q^{n-k}-1}{q^k -1} \bigg) {\ n-1 \brack k-1}_q \\
&=& q^k {\ n \brack k-1}_q.
\end{eqnarray*}
In the same fashion one can prove the property of \emph{symmetry} for
Gaussian polynomials, 
\begin{eqnarray}
{\ n \brack k}_q = {\ n \brack n-k}_q  \label{symmetry}
\end{eqnarray}
as well as an analogy of the property of \emph{absorption} for binomial
coefficients, 
\begin{eqnarray}
{\ n \brack k}_q = \frac{1-q^n}{1-q^{n-k}}{\ n-1 \brack k-1}_q.
\end{eqnarray}

A well known combinatorial interpretaion of Gaussian polynomial ${\ n\brack k%
}_{q}$ is that it enumerates the number of \emph{integer partitions} that
fit into the \emph{square lattice} region of size $k\times (n-k)$ (see the
book \cite{azose} by J. Azose, A. Benjamin and K. Garrett). More precisely,
the Gaussian polynomial can be defined as the \emph{generating function} for
this type of partitions. For example, consider the polynomial 
\begin{equation*}
{\ 4\brack2}_{q}=q^{4}+q^{3}+2q^{2}+q+1.
\end{equation*}%
This means that there are six distinct partitions whose \emph{Young diagrams}
fit into a $2\times 2$ lattice region and that these partitions are 
\begin{equation*}
(2,2),(2,1),(2),(1,1),(1),(0)
\end{equation*}%
(there are 2 partitions of the number 2 while the other partitions are of
numbers 4, 3, and 1. Now the symmetry property (\ref{symmetry}) follows
immediately by conjugation of the Young diagrams. Similar bijective proofs
can also be done for recurences (\ref{thefirstrecurrence}) and (\ref%
{thesecondrecurrence}) (see the book \cite{AndrewsEricsson} by G. Andrews
and K. Ericcson).

Obviously, instead by partitions, one can interpret Gaussian polynomial by
lattice paths from $(0,0)$ to $(k,n-k)$, where an increment by 1 right has
the weight of 1 while an increment by 1 up has the weight of $q^{s}$ where $s
$ is the number of previous horizontal increments. Furthermore, we establish
a bijection between these paths and a \emph{board tilings} by squares and
dominoes as follows:

\begin{description}
\item \textit{i)} we code an increment by 1 right in a path by a domino, and

\item \textit{ii)} we code an increment by 1 up by a square.
\end{description}

We illustrate this 1 to 1 correspondence in Figure \ref{fig0}. The reasoning
above proves the next lemma.

\begin{lemma}
\label{lemma1} The Gausian polynomial ${\ n\brack k}_q$ represents the
number of $(n+k)$-board tilings by $n-k$ squares and $k$ dominoes where the
weight of every domino is equal to 1 while the weight of a square is equal
to $q^s$, with $s$ being the number of preceding dominoes.
\end{lemma}

\begin{figure}[h!]
\setlength{\unitlength}{0.5cm}
\par
\begin{picture}(0,4)(-3.5,0)

\put(0.1,0){\line(1,0){0.9}}
\put(1,0){\line(0,1){1}}
\put(1,1){\line(1,0){2}}
\put(3,1){\line(0,1){1}}
\put(3,2){\line(1,0){1}}
\put(4,2){\line(0,1){0.9}}

 \put(10,0){\line(1,0){11}}
 \put(10,1){\line(1,0){11}}
 \put(10,0){\line(0,1){1}}
 \put(12,0){\line(0,1){1}}
 \put(13,0){\line(0,1){1}}
 \put(15,0){\line(0,1){1}}
 \put(17,0){\line(0,1){1}}
 \put(18,0){\line(0,1){1}}
 \put(20,0){\line(0,1){1}}
 \put(21,0){\line(0,1){1}}

\put(0,0){\circle{0.2}}
\put(4,3){\circle{0.2}}

\put(6,1){\vector(1,0){2}}
\put(8,0.5){\vector(-1,0){2}}

\put(4.1, 2.3){{ $q^8$ }}
\put(12.1, 1.3){{ $q^1$ }}
\put(17.1, 1.3){{ $q^3$ }}
\put(20.1, 1.3){{ $q^4$ }}

\end{picture}
\caption{There is a 1 to 1 correspondence between weighted lattice paths
with $k$ horizontal increments and weighted board tilings with $k$ dominoes.}
\label{fig0}
\end{figure}
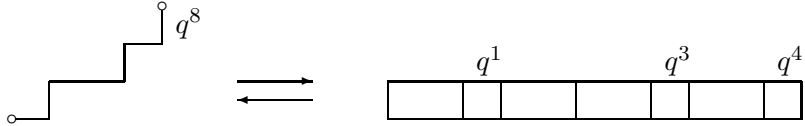

Note that the parameter $n$ in Lemma \ref{lemma1} represents the number of
parts in a tiling (the number of increments of length 1 in a lattice path).
Figure \ref{fig1} shows $6$-board tilings that correspond to partitions in
the example above. According to Lemma \ref{lemma1}, these tilings are
represented by the Gaussian polynomial ${4\brack2}_{q}$. In what follows we
use Lemma \ref{lemma1} to prove some Gaussian polynomial identities.

\begin{figure}[h!]
\setlength{\unitlength}{0.5cm}
\par
\begin{picture}(0,4)(-3,0)
\put(0,2){\line(1,0){6}}
\put(0,2){\line(0,1){1}}
\put(2,2){\line(0,1){1}}
\put(4,2){\line(0,1){1}}
\put(5,2){\line(0,1){1}}
\put(6,2){\line(0,1){1}}
\put(0,3){\line(1,0){6}}

\put(8,2){\line(1,0){6}}
\put(8,2){\line(0,1){1}}
\put(9,2){\line(0,1){1}}
\put(11,2){\line(0,1){1}}
\put(13,2){\line(0,1){1}}
\put(14,2){\line(0,1){1}}
\put(8,3){\line(1,0){6}}

\put(16,2){\line(1,0){6}}
\put(16,2){\line(0,1){1}}
\put(17,2){\line(0,1){1}}
\put(19,2){\line(0,1){1}}
\put(20,2){\line(0,1){1}}
\put(22,2){\line(0,1){1}}
\put(16,3){\line(1,0){6}}

\put(0,0){\line(1,0){6}}
\put(0,0){\line(0,1){1}}
\put(2,0){\line(0,1){1}}
\put(3,0){\line(0,1){1}}
\put(5,0){\line(0,1){1}}
\put(6,0){\line(0,1){1}}
\put(0,1){\line(1,0){6}}

\put(8,0){\line(1,0){6}}
\put(8,0){\line(0,1){1}}
\put(10,0){\line(0,1){1}}
\put(11,0){\line(0,1){1}}
\put(12,0){\line(0,1){1}}
\put(14,0){\line(0,1){1}}
\put(8,1){\line(1,0){6}}

\put(16,0){\line(1,0){6}}
\put(16,0){\line(0,1){1}}
\put(17,0){\line(0,1){1}}
\put(18,0){\line(0,1){1}}
\put(20,0){\line(0,1){1}}
\put(22,0){\line(0,1){1}}
\put(16,1){\line(1,0){6}}

\put(6.1, 2.3){{ $q^4$ }}
\put(14.1, 2.3){{ $q^2$ }}
\put(22.1, 2.3){{ $q^1$ }}

\put(6.1, 0.3){{ $q^3$ }}
\put(14.1, 0.3){{ $q^2$ }}
\put(22.1, 0.3){{ $1$ }}

\end{picture}
\caption{All six tilings of a board of length 6 having two dominoes, which
are enumerated by the polynomial $q^4 + q^3 + 2q^2 + q + 1$.}
\label{fig1}
\end{figure}
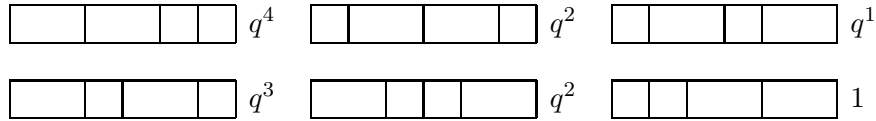

\section{Two pairs of polynomial identities}

We shall now present obtained identities for the Gaussian polynomials, in
the following statements. To get these, we enumerate elements of the set of
all $n$-board weighted tilings in respect to certain criterions.

\begin{thm}
\label{thm1} For natural numbers $n$ and $k$, where $n \ge k$, we have 
\begin{eqnarray*}
\sum_{i=0}^{k} q^{i} {n-k-1 +i \brack i}_q = {\ n \brack k}_q.
\end{eqnarray*}
\end{thm}

\begin{proof}
We consider weighted $(n+k)$-board tilings with $k$ dominoes and $n-k$ squares, in respect to the last square in a tiling. The set $S$ of all such tilings we separate into $k$ disjoint subsets $S_1, S_2, \ldots, S_k$ such that in set $S_1$ there are tilings having the last square on the position $n+k-1$. In $S_2$ there are tilings having the last square on position $n+k-3$ (meaning that there are $k-1$ dominoes left from the last square and a sole domino is placed right from the last square). Furthermore, in $S_3$ the last square is on position $n+k-5$, etc. 

Now, the tilings in $S_1$ are represened by $${ n-1\brack k}_q  q^k.$$ Namely, by Lemma \ref{lemma1}, tilings of the length $n+k-1$ are represented by ${ n-1\brack k}_q$ and the weight of the last square is $q^k$ (since there are $k$ dominoes placed on the left of the last square). Furthermore, by an analogue argument, tilings in $S_2$ are represented by $${ n-2 \brack k-1}_q  q^{k-1},$$ etc. Finally, tilings in the set $S_k$ are represented by $${ n-k-1 \brack 0}_q,$$ since there are no dominoes on the left of the last square. Having in mind that the sets $S_i$ are disjoint, the sum
\begin{eqnarray*}
{ n-1 \brack k}_q  q^{k} + { n-2 \brack k-1}_q  q^{k-1} + \cdots + { n-k-1 \brack 0}_q
\end{eqnarray*}
represents the set $S$ and this fact completes the proof.

\end{proof}

For the purpose of proving the following theorem, again we consider weighted 
$(n+k)$-board tilings with $k$ dominoes and $n-k$ squares, but now with
respect to the last domino in a tiling.

\begin{thm}
\label{thm2} For natural numbers $n$ and $k$, where $n \ge k$, we have 
\begin{eqnarray*}
\sum_{i=0}^{n-k} q^{k(n-k-i)} {k -1 +i \brack k-1}_q = {\ n \brack k}_q.
\end{eqnarray*}
\end{thm}

\begin{proof}
The set $S$ of weighted tilings of length $n+k$ we decompose into partition of $n-k+1$ disjoint subsets $S_0, S_1, S_2, \ldots , S_{n-k}$, such that in the subset $S_i$ there are exactly $n-k-i$ squares on the right hand side of the last domino in a tiling.

The set $S_0$ is represented by $${ k-1 \brack k-1 }_q q^{k(n-k)}.$$ This holds true since on the left hand side from the last domino in $S_0$ there are $k-1$ dominoes and from the fact that there are $n-k$ squares on the right hand side of the last domino - in every of these tilings. Tilings in the set $S_1$ have $n-r-1$ dominoes on the right hand side of the last domino which means that they are enumerated by $${ k \brack k-1 }_q q^{k(n-k-1)},$$ etc. Finally, tilings in the set $S_{n-k}$ are represented by the polynomial $${ n-1 \brack k-1 }_q. $$ Thus, the sum of these terms,
\begin{eqnarray*}
{ k-1 \brack k-1}_q   q^{k(n-k)} + { k \brack k-1}_q  q^{k(n-k-1)}+ \cdots + { n-1 \brack k-1}_q
\end{eqnarray*}
is equal to the polynomial ${ n \brack k }_q$.
\end{proof}

Figure \ref{fig2} shows tilings of length 7 with two dominoes and three
squares separated into subsets $S_{0},S_{1},S_{2},S_{3}$, with respect to
the last domino. One can easily establish that $%
|S_{0}|=1,|S_{1}|=2,|S_{2}|=3,|S_{3}|=4$ and that polynomial ${\ 5\brack2}%
_{q}$ enumerates these tilings.

\begin{figure}[h]
\setlength{\unitlength}{0.5cm}
\par
\begin{picture}(0,8)(-10,0)

\put(0,6){\line(1,0){7}}
\put(0,6){\line(0,1){1}}
\put(2,6){\line(0,1){1}}
\put(4,6){\line(0,1){1}}
\put(5,6){\line(0,1){1}}
\put(6,6){\line(0,1){1}}
\put(7,6){\line(0,1){1}}
\put(0,7){\line(1,0){7}}

\put(0,4){\line(1,0){7}}
\put(0,4){\line(0,1){1}}
\put(3,4){\line(0,1){1}}
\put(5,4){\line(0,1){1}}
\put(6,4){\line(0,1){1}}
\put(7,4){\line(0,1){1}}
\put(0,5){\line(1,0){7}}

\put(0,2){\line(1,0){7}}
\put(0,2){\line(0,1){1}}
\put(4,2){\line(0,1){1}}
\put(6,2){\line(0,1){1}}
\put(7,2){\line(0,1){1}}
\put(0,3){\line(1,0){7}}

\put(0,0){\line(1,0){7}}
\put(0,0){\line(0,1){1}}
\put(5,0){\line(0,1){1}}
\put(7,0){\line(0,1){1}}
\put(0,1){\line(1,0){7}}

\put(7.2, 6.3){{ $S_0$ }}
\put(7.2, 4.3){{ $S_1$ }}
\put(7.2, 2.3){{ $S_2$ }}
\put(7.2, 0.3){{ $S_3$ }}

\multiput(2,6)(0.01,0){20}{\line(0,1){0.2}}
\multiput(2.4,6)(0.01,0){20}{\line(0,1){0.2}}
\multiput(2.8,6)(0.01,0){20}{\line(0,1){0.2}}
\multiput(3.2,6)(0.01,0){20}{\line(0,1){0.2}}
\multiput(3.6,6)(0.01,0){20}{\line(0,1){0.2}}

\multiput(2.2,6.2)(0.01,0){20}{\line(0,1){0.2}}
\multiput(2.6,6.2)(0.01,0){20}{\line(0,1){0.2}}
\multiput(3,6.2)(0.01,0){20}{\line(0,1){0.2}}
\multiput(3.4,6.2)(0.01,0){20}{\line(0,1){0.2}}
\multiput(3.8,6.2)(0.01,0){20}{\line(0,1){0.2}}

\multiput(2,6.4)(0.01,0){20}{\line(0,1){0.2}}
\multiput(2.4,6.4)(0.01,0){20}{\line(0,1){0.2}}
\multiput(2.8,6.4)(0.01,0){20}{\line(0,1){0.2}}
\multiput(3.2,6.4)(0.01,0){20}{\line(0,1){0.2}}
\multiput(3.6,6.4)(0.01,0){20}{\line(0,1){0.2}}

\multiput(2.2,6.6)(0.01,0){20}{\line(0,1){0.2}}
\multiput(2.6,6.6)(0.01,0){20}{\line(0,1){0.2}}
\multiput(3,6.6)(0.01,0){20}{\line(0,1){0.2}}
\multiput(3.4,6.6)(0.01,0){20}{\line(0,1){0.2}}
\multiput(3.8,6.6)(0.01,0){20}{\line(0,1){0.2}}

\multiput(2,6.8)(0.01,0){20}{\line(0,1){0.2}}
\multiput(2.4,6.8)(0.01,0){20}{\line(0,1){0.2}}
\multiput(2.8,6.8)(0.01,0){20}{\line(0,1){0.2}}
\multiput(3.2,6.8)(0.01,0){20}{\line(0,1){0.2}}
\multiput(3.6,6.8)(0.01,0){20}{\line(0,1){0.2}}

\multiput(3,4)(0.01,0){20}{\line(0,1){0.2}}
\multiput(3.4,4)(0.01,0){20}{\line(0,1){0.2}}
\multiput(3.8,4)(0.01,0){20}{\line(0,1){0.2}}
\multiput(4.2,4)(0.01,0){20}{\line(0,1){0.2}}
\multiput(4.6,4)(0.01,0){20}{\line(0,1){0.2}}

\multiput(3.2,4.2)(0.01,0){20}{\line(0,1){0.2}}
\multiput(3.6,4.2)(0.01,0){20}{\line(0,1){0.2}}
\multiput(4,4.2)(0.01,0){20}{\line(0,1){0.2}}
\multiput(4.4,4.2)(0.01,0){20}{\line(0,1){0.2}}
\multiput(4.8,4.2)(0.01,0){20}{\line(0,1){0.2}}

\multiput(3,4.4)(0.01,0){20}{\line(0,1){0.2}}
\multiput(3.4,4.4)(0.01,0){20}{\line(0,1){0.2}}
\multiput(3.8,4.4)(0.01,0){20}{\line(0,1){0.2}}
\multiput(4.2,4.4)(0.01,0){20}{\line(0,1){0.2}}
\multiput(4.6,4.4)(0.01,0){20}{\line(0,1){0.2}}

\multiput(3.2,4.6)(0.01,0){20}{\line(0,1){0.2}}
\multiput(3.6,4.6)(0.01,0){20}{\line(0,1){0.2}}
\multiput(4,4.6)(0.01,0){20}{\line(0,1){0.2}}
\multiput(4.4,4.6)(0.01,0){20}{\line(0,1){0.2}}
\multiput(4.8,4.6)(0.01,0){20}{\line(0,1){0.2}}

\multiput(3,4.8)(0.01,0){20}{\line(0,1){0.2}}
\multiput(3.4,4.8)(0.01,0){20}{\line(0,1){0.2}}
\multiput(3.8,4.8)(0.01,0){20}{\line(0,1){0.2}}
\multiput(4.2,4.8)(0.01,0){20}{\line(0,1){0.2}}
\multiput(4.6,4.8)(0.01,0){20}{\line(0,1){0.2}}

\multiput(4,2)(0.01,0){20}{\line(0,1){0.2}}
\multiput(4.4,2)(0.01,0){20}{\line(0,1){0.2}}
\multiput(4.8,2)(0.01,0){20}{\line(0,1){0.2}}
\multiput(5.2,2)(0.01,0){20}{\line(0,1){0.2}}
\multiput(5.6,2)(0.01,0){20}{\line(0,1){0.2}}

\multiput(4.2,2.2)(0.01,0){20}{\line(0,1){0.2}}
\multiput(4.6,2.2)(0.01,0){20}{\line(0,1){0.2}}
\multiput(5,2.2)(0.01,0){20}{\line(0,1){0.2}}
\multiput(5.4,2.2)(0.01,0){20}{\line(0,1){0.2}}
\multiput(5.8,2.2)(0.01,0){20}{\line(0,1){0.2}}

\multiput(4,2.4)(0.01,0){20}{\line(0,1){0.2}}
\multiput(4.4,2.4)(0.01,0){20}{\line(0,1){0.2}}
\multiput(4.8,2.4)(0.01,0){20}{\line(0,1){0.2}}
\multiput(5.2,2.4)(0.01,0){20}{\line(0,1){0.2}}
\multiput(5.6,2.4)(0.01,0){20}{\line(0,1){0.2}}

\multiput(4.2,2.6)(0.01,0){20}{\line(0,1){0.2}}
\multiput(4.6,2.6)(0.01,0){20}{\line(0,1){0.2}}
\multiput(5,2.6)(0.01,0){20}{\line(0,1){0.2}}
\multiput(5.4,2.6)(0.01,0){20}{\line(0,1){0.2}}
\multiput(5.8,2.6)(0.01,0){20}{\line(0,1){0.2}}

\multiput(4,2.8)(0.01,0){20}{\line(0,1){0.2}}
\multiput(4.4,2.8)(0.01,0){20}{\line(0,1){0.2}}
\multiput(4.8,2.8)(0.01,0){20}{\line(0,1){0.2}}
\multiput(5.2,2.8)(0.01,0){20}{\line(0,1){0.2}}
\multiput(5.6,2.8)(0.01,0){20}{\line(0,1){0.2}}

\multiput(5,0)(0.01,0){20}{\line(0,1){0.2}}
\multiput(5.4,0)(0.01,0){20}{\line(0,1){0.2}}
\multiput(5.8,0)(0.01,0){20}{\line(0,1){0.2}}
\multiput(6.2,0)(0.01,0){20}{\line(0,1){0.2}}
\multiput(6.6,0)(0.01,0){20}{\line(0,1){0.2}}

\multiput(5.2,0.2)(0.01,0){20}{\line(0,1){0.2}}
\multiput(5.6,0.2)(0.01,0){20}{\line(0,1){0.2}}
\multiput(6,0.2)(0.01,0){20}{\line(0,1){0.2}}
\multiput(6.4,0.2)(0.01,0){20}{\line(0,1){0.2}}
\multiput(6.8,0.2)(0.01,0){20}{\line(0,1){0.2}}

\multiput(5,0.4)(0.01,0){20}{\line(0,1){0.2}}
\multiput(5.4,0.4)(0.01,0){20}{\line(0,1){0.2}}
\multiput(5.8,0.4)(0.01,0){20}{\line(0,1){0.2}}
\multiput(6.2,0.4)(0.01,0){20}{\line(0,1){0.2}}
\multiput(6.6,0.4)(0.01,0){20}{\line(0,1){0.2}}

\multiput(5.2,0.6)(0.01,0){20}{\line(0,1){0.2}}
\multiput(5.6,0.6)(0.01,0){20}{\line(0,1){0.2}}
\multiput(6,0.6)(0.01,0){20}{\line(0,1){0.2}}
\multiput(6.4,0.6)(0.01,0){20}{\line(0,1){0.2}}
\multiput(6.8,0.6)(0.01,0){20}{\line(0,1){0.2}}

\multiput(5,0.8)(0.01,0){20}{\line(0,1){0.2}}
\multiput(5.4,0.8)(0.01,0){20}{\line(0,1){0.2}}
\multiput(5.8,0.8)(0.01,0){20}{\line(0,1){0.2}}
\multiput(6.2,0.8)(0.01,0){20}{\line(0,1){0.2}}
\multiput(6.6,0.8)(0.01,0){20}{\line(0,1){0.2}}

\end{picture}
\caption{Tilings of length 7 with two dominoes and three squares separated
into four subsets, with respect to the last domino in a tiling.}
\label{fig2}
\end{figure}
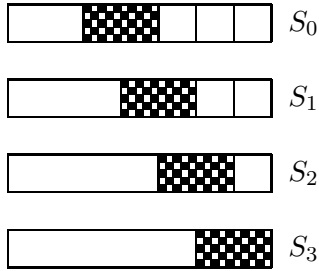

In order to prove Theorem \ref{thm3} we consider a class of tilings with an
odd number of dominoes. More precisely, we take into account weighted $%
(n+2r+1)$-board tilings with $2r+1$ dominoes. The set $S$ of all such
tilings we separate into $n-2r$ subsets $S_{1},S_{2},\ldots ,S_{n-2r}$, by
means of the median domino. Tilings in set $S_{i}$ have the median domino
covering cells $2r+i$ and $2r+i+1$. Thus, in $S_{1}$ there are tilings
having the median domino covering cells $2r+1$ and $2r+2$. This means that
on the left hand side of the median domino there are $r$ dominoes covering
cells 1 through $2r$. Now we separate a tiling into parts $t_{1}$ - from
cells 1 through $2r$ and $t_{2}$ - from cells $2r+3$ through $n+2r+1$.
Clearly, the tilings in the set $t_{1}$ are enumerated by ${\ r\brack r}_{q}$%
. The tilings in $t_{2}$ are enumerated by 
\begin{equation*}
{\ n-r-1\brack r}_{q}q^{(r+1)(n-2r-1)},
\end{equation*}%
where the factor $q^{(r+1)(n-2r-1)}$ is added because there are $r+1$
dominoes left from a tiling in $t_{2}$ and there are $n-2r-1$ squares in a
such tiling. By the same arguments one can conclude that the tilings in $%
S_{2}$ are enumerated by the polynomial 
\begin{equation*}
{\ r+1\brack r}_{q}{\ n-r\brack r}_{q}q^{(r+1)(n-2r-2)},
\end{equation*}%
tilings in $S_{3}$ by 
\begin{equation*}
{\ r+2\brack r}_{q}{\ n-r+1\brack r}_{q}q^{(r+1)(n-2r-3)},
\end{equation*}%
\ldots , and tilings in $S_{n-2r}$ by 
\begin{equation*}
{n-r-1\brack r}_{q}{\ r\brack r}_{q}.
\end{equation*}%
Finally, these terms sum up to 
\begin{equation*}
{\ r\brack r}_{q}{\ n-r-1\brack r}_{q}q^{(r+1)(n-2r-1)}+\cdots +{n-r-1\brack %
r}_{q}{\ r\brack r}_{q}
\end{equation*}
which is the polynomial ${\ n\brack2r+1}_{q}$. This reasoning completes the
proof of Theorem \ref{thm3}.

\begin{thm}
\label{thm3} For natural numbers $n$ and $r$, where $n \ge 2r+1$, we have 
\begin{eqnarray*}
\sum_{i=1}^{n-2r} q^{(r+1)(n-2r-i)} {r-1+i \brack r}_q {n- r-i \brack r}_q = 
{\ n \brack 2r+1}_q.
\end{eqnarray*}
\end{thm}

As an illustration of Theorem \ref{thm3} we can take that there are 20
tilings of the length 9, tiled by three dominoes and three squares. These
can be separated by means of the median domino into sets $S_{1},S_{2},S_{3}$
and $S_{4}$, where $|S_{1}|=4,|S_{2}|=6,|S_{3}|=6,|S_{4}|=4$ (Figure \ref{fig4}). These
cardinalities correspond to the terms in the equality stated in Theorem \ref%
{thm3} for $n=6$ and $r=1$. The Gaussian polynomial ${6\brack3}_{q}$
representing these tilings reads as 
\begin{equation*}
q^{9}+q^{8}+2q^{7}+3q^{6}+3q^{5}+3q^{4}+3q^{3}+2q^{2}+q+1.
\end{equation*}

\begin{figure}[h]
\setlength{\unitlength}{0.5cm}
\par
\begin{picture}(0,8)(-10,0)

\put(0,6){\line(1,0){9}}
\put(0,6){\line(0,1){1}}
\put(2,6){\line(0,1){1}}
\put(4,6){\line(0,1){1}}
\put(6,6){\line(0,1){1}}
\put(7,6){\line(0,1){1}}
\put(8,6){\line(0,1){1}}
\put(9,6){\line(0,1){1}}
\put(0,7){\line(1,0){9}}

\put(0,4){\line(1,0){9}}
\put(0,4){\line(0,1){1}}
\put(3,4){\line(0,1){1}}
\put(5,4){\line(0,1){1}}
\put(2,4){\line(0,1){1}}
\put(6,4){\line(0,1){1}}
\put(7,4){\line(0,1){1}}
\put(9,4){\line(0,1){1}}
\put(0,5){\line(1,0){9}}

\put(0,2){\line(1,0){9}}
\put(0,2){\line(0,1){1}}
\put(4,2){\line(0,1){1}}
\put(2,2){\line(0,1){1}}
\put(3,2){\line(0,1){1}}
\put(6,2){\line(0,1){1}}
\put(7,2){\line(0,1){1}}
\put(9,2){\line(0,1){1}}
\put(0,3){\line(1,0){9}}

\put(0,0){\line(1,0){9}}
\put(0,0){\line(0,1){1}}
\put(1,0){\line(0,1){1}}
\put(3,0){\line(0,1){1}}
\put(4,0){\line(0,1){1}}
\put(5,0){\line(0,1){1}}
\put(9,0){\line(0,1){1}}
\put(7,0){\line(0,1){1}}
\put(0,1){\line(1,0){9}}


\multiput(2,6)(0.01,0){20}{\line(0,1){0.2}}
\multiput(2.4,6)(0.01,0){20}{\line(0,1){0.2}}
\multiput(2.8,6)(0.01,0){20}{\line(0,1){0.2}}
\multiput(3.2,6)(0.01,0){20}{\line(0,1){0.2}}
\multiput(3.6,6)(0.01,0){20}{\line(0,1){0.2}}

\multiput(2.2,6.2)(0.01,0){20}{\line(0,1){0.2}}
\multiput(2.6,6.2)(0.01,0){20}{\line(0,1){0.2}}
\multiput(3,6.2)(0.01,0){20}{\line(0,1){0.2}}
\multiput(3.4,6.2)(0.01,0){20}{\line(0,1){0.2}}
\multiput(3.8,6.2)(0.01,0){20}{\line(0,1){0.2}}

\multiput(2,6.4)(0.01,0){20}{\line(0,1){0.2}}
\multiput(2.4,6.4)(0.01,0){20}{\line(0,1){0.2}}
\multiput(2.8,6.4)(0.01,0){20}{\line(0,1){0.2}}
\multiput(3.2,6.4)(0.01,0){20}{\line(0,1){0.2}}
\multiput(3.6,6.4)(0.01,0){20}{\line(0,1){0.2}}

\multiput(2.2,6.6)(0.01,0){20}{\line(0,1){0.2}}
\multiput(2.6,6.6)(0.01,0){20}{\line(0,1){0.2}}
\multiput(3,6.6)(0.01,0){20}{\line(0,1){0.2}}
\multiput(3.4,6.6)(0.01,0){20}{\line(0,1){0.2}}
\multiput(3.8,6.6)(0.01,0){20}{\line(0,1){0.2}}

\multiput(2,6.8)(0.01,0){20}{\line(0,1){0.2}}
\multiput(2.4,6.8)(0.01,0){20}{\line(0,1){0.2}}
\multiput(2.8,6.8)(0.01,0){20}{\line(0,1){0.2}}
\multiput(3.2,6.8)(0.01,0){20}{\line(0,1){0.2}}
\multiput(3.6,6.8)(0.01,0){20}{\line(0,1){0.2}}

\multiput(3,4)(0.01,0){20}{\line(0,1){0.2}}
\multiput(3.4,4)(0.01,0){20}{\line(0,1){0.2}}
\multiput(3.8,4)(0.01,0){20}{\line(0,1){0.2}}
\multiput(4.2,4)(0.01,0){20}{\line(0,1){0.2}}
\multiput(4.6,4)(0.01,0){20}{\line(0,1){0.2}}

\multiput(3.2,4.2)(0.01,0){20}{\line(0,1){0.2}}
\multiput(3.6,4.2)(0.01,0){20}{\line(0,1){0.2}}
\multiput(4,4.2)(0.01,0){20}{\line(0,1){0.2}}
\multiput(4.4,4.2)(0.01,0){20}{\line(0,1){0.2}}
\multiput(4.8,4.2)(0.01,0){20}{\line(0,1){0.2}}

\multiput(3,4.4)(0.01,0){20}{\line(0,1){0.2}}
\multiput(3.4,4.4)(0.01,0){20}{\line(0,1){0.2}}
\multiput(3.8,4.4)(0.01,0){20}{\line(0,1){0.2}}
\multiput(4.2,4.4)(0.01,0){20}{\line(0,1){0.2}}
\multiput(4.6,4.4)(0.01,0){20}{\line(0,1){0.2}}

\multiput(3.2,4.6)(0.01,0){20}{\line(0,1){0.2}}
\multiput(3.6,4.6)(0.01,0){20}{\line(0,1){0.2}}
\multiput(4,4.6)(0.01,0){20}{\line(0,1){0.2}}
\multiput(4.4,4.6)(0.01,0){20}{\line(0,1){0.2}}
\multiput(4.8,4.6)(0.01,0){20}{\line(0,1){0.2}}

\multiput(3,4.8)(0.01,0){20}{\line(0,1){0.2}}
\multiput(3.4,4.8)(0.01,0){20}{\line(0,1){0.2}}
\multiput(3.8,4.8)(0.01,0){20}{\line(0,1){0.2}}
\multiput(4.2,4.8)(0.01,0){20}{\line(0,1){0.2}}
\multiput(4.6,4.8)(0.01,0){20}{\line(0,1){0.2}}

\multiput(4,2)(0.01,0){20}{\line(0,1){0.2}}
\multiput(4.4,2)(0.01,0){20}{\line(0,1){0.2}}
\multiput(4.8,2)(0.01,0){20}{\line(0,1){0.2}}
\multiput(5.2,2)(0.01,0){20}{\line(0,1){0.2}}
\multiput(5.6,2)(0.01,0){20}{\line(0,1){0.2}}

\multiput(4.2,2.2)(0.01,0){20}{\line(0,1){0.2}}
\multiput(4.6,2.2)(0.01,0){20}{\line(0,1){0.2}}
\multiput(5,2.2)(0.01,0){20}{\line(0,1){0.2}}
\multiput(5.4,2.2)(0.01,0){20}{\line(0,1){0.2}}
\multiput(5.8,2.2)(0.01,0){20}{\line(0,1){0.2}}

\multiput(4,2.4)(0.01,0){20}{\line(0,1){0.2}}
\multiput(4.4,2.4)(0.01,0){20}{\line(0,1){0.2}}
\multiput(4.8,2.4)(0.01,0){20}{\line(0,1){0.2}}
\multiput(5.2,2.4)(0.01,0){20}{\line(0,1){0.2}}
\multiput(5.6,2.4)(0.01,0){20}{\line(0,1){0.2}}

\multiput(4.2,2.6)(0.01,0){20}{\line(0,1){0.2}}
\multiput(4.6,2.6)(0.01,0){20}{\line(0,1){0.2}}
\multiput(5,2.6)(0.01,0){20}{\line(0,1){0.2}}
\multiput(5.4,2.6)(0.01,0){20}{\line(0,1){0.2}}
\multiput(5.8,2.6)(0.01,0){20}{\line(0,1){0.2}}

\multiput(4,2.8)(0.01,0){20}{\line(0,1){0.2}}
\multiput(4.4,2.8)(0.01,0){20}{\line(0,1){0.2}}
\multiput(4.8,2.8)(0.01,0){20}{\line(0,1){0.2}}
\multiput(5.2,2.8)(0.01,0){20}{\line(0,1){0.2}}
\multiput(5.6,2.8)(0.01,0){20}{\line(0,1){0.2}}

\multiput(5,0)(0.01,0){20}{\line(0,1){0.2}}
\multiput(5.4,0)(0.01,0){20}{\line(0,1){0.2}}
\multiput(5.8,0)(0.01,0){20}{\line(0,1){0.2}}
\multiput(6.2,0)(0.01,0){20}{\line(0,1){0.2}}
\multiput(6.6,0)(0.01,0){20}{\line(0,1){0.2}}

\multiput(5.2,0.2)(0.01,0){20}{\line(0,1){0.2}}
\multiput(5.6,0.2)(0.01,0){20}{\line(0,1){0.2}}
\multiput(6,0.2)(0.01,0){20}{\line(0,1){0.2}}
\multiput(6.4,0.2)(0.01,0){20}{\line(0,1){0.2}}
\multiput(6.8,0.2)(0.01,0){20}{\line(0,1){0.2}}

\multiput(5,0.4)(0.01,0){20}{\line(0,1){0.2}}
\multiput(5.4,0.4)(0.01,0){20}{\line(0,1){0.2}}
\multiput(5.8,0.4)(0.01,0){20}{\line(0,1){0.2}}
\multiput(6.2,0.4)(0.01,0){20}{\line(0,1){0.2}}
\multiput(6.6,0.4)(0.01,0){20}{\line(0,1){0.2}}

\multiput(5.2,0.6)(0.01,0){20}{\line(0,1){0.2}}
\multiput(5.6,0.6)(0.01,0){20}{\line(0,1){0.2}}
\multiput(6,0.6)(0.01,0){20}{\line(0,1){0.2}}
\multiput(6.4,0.6)(0.01,0){20}{\line(0,1){0.2}}
\multiput(6.8,0.6)(0.01,0){20}{\line(0,1){0.2}}

\multiput(5,0.8)(0.01,0){20}{\line(0,1){0.2}}
\multiput(5.4,0.8)(0.01,0){20}{\line(0,1){0.2}}
\multiput(5.8,0.8)(0.01,0){20}{\line(0,1){0.2}}
\multiput(6.2,0.8)(0.01,0){20}{\line(0,1){0.2}}
\multiput(6.6,0.8)(0.01,0){20}{\line(0,1){0.2}}

\end{picture}
\caption{Representatives of sets $S_1, \ldots, S_4$, respectively, of tilings of length 9 with three dominoes, with the central domino marked.}
\label{fig4}
\end{figure}
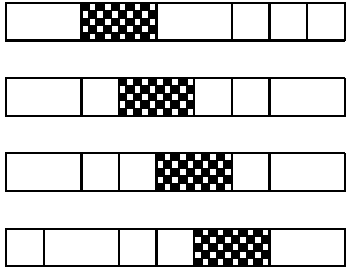

Having in mind the symmetry of pairs of factors under the sum in Theorem \ref%
{thm3}, we have a simple consequence to binomial coefficients. By the
substitution $m=n-2r$ we immediately obtain the following corollary.

\begin{cor}
\label{cor1} For natural numbers $n$ and $r$, where $n$ is even we have 
\begin{eqnarray*}
2 \sum_{i=1}^{n/2} \binom{r-1+i}{r} \binom{n+r-i}{r} = \binom{n+2r}{2r+1}.
\end{eqnarray*}
\end{cor}

In order to prove our next result, we consider the class of $(n+r)$-board
tilings having $r$ dominoes where the constraint on the length of a tiling
is that $n-r$ must be odd. Thus, the number of squares in a tiling is odd
which provides us a double counting of tilings. By Lemma \ref{lemma1}, the
coefficient ${n\brack r}_{q}$ enumerates these tilings. On the other hand,
the sum 
\begin{eqnarray*}
&&{\ \frac{n-r-1}{2}\brack0}_{q}{\ \frac{n-r-1}{2}+r\brack k}_{q}+{\ \frac{%
n-r-1}{2}+1\brack1}_{q}{\ \frac{n-r-1}{2}+r-1\brack k-1}_{q}q^{(n-r+1)/2}+%
\cdots + \\
&&{\ \frac{n-r-1}{2}+r\brack k}_{q}{\ \frac{n-r-1}{2}\brack0}%
_{q}q^{r(n-r+1)/2}
\end{eqnarray*}%
consisting of $r+1$ terms also enumerates them. This follows by analogue
arguments as in the previous proof, with the difference that here we
enumerate tilings with respect to the median square (which always exists
since their number is odd). Now, by the substitution $m:=n-r,$ we have the
following statement.

\begin{thm}
\label{thm4} For natural numbers $n$ and $r$, where $n$ is odd, we have 
\begin{equation*}
\sum_{i=0}^{r}q^{i(n+1)/2}{\frac{n-1}{2}+i\brack i}_{q}{\ \frac{n-1}{2}+r-i%
\brack r-i}_{q}={\ n+r\brack r}_{q}.
\end{equation*}
\end{thm}

In the same fashion as in case of Theorem \ref{thm3} we consider, as an
illustration of Theorem \ref{thm4}, tilings of the length 13. The Gaussian
polynomial ${9\brack4}_{q}$ representing these tilings reads as 
\begin{eqnarray*}
&&q^{20}+q^{19}+2q^{18}+3q^{17}+5q^{16}+6q^{15}+8q^{14}+9q^{13}+11q^{12}+11q^{11}+
\\
&&12q^{10}+11q^{9}+11q^{8}+9q^{7}+8q^{6}+6q^{5}+5q^{4}+3q^{3}+2q^{2}+q+1.
\end{eqnarray*}

Again we have a consequence to binomial coefficients and we state it in
Corollary \ref{cor2}.

\begin{cor}
\label{cor2} For odd natural numbers $n$ and $r$ we have 
\begin{eqnarray*}
2 \sum_{i=0}^{(r+1)/2} \binom{\frac{n-1}{2} +i }{i} \binom{ \frac{n-1}{2} +r
-i}{r-i} = \binom{n+r}{r}.
\end{eqnarray*}
\end{cor}

\section{Concluding remarks}

In this paper we use a combinatorial interpretation of $q$-binomial
coefficients through the weighted lattice paths to establish identities for
these polynomials. In particular, we establish two duals of Gaussian
polynomial identities. A pair of identities is stated in Theorems \ref{thm1}
and \ref{thm2}. In Theorems \ref{thm3} and \ref{thm4} we evaluate the sum
with terms of two factors.

We believe that ideas and results presented here could be used to get
further Gaussian polynomial identities. Here we enumerate weighted tilings
with respect to position of the last square, the last domino, the median
domino and the median square. We are fairly convinced that one can found
other similar criteria and employ them to obtain further identities. More
refined constraints on tilings could possibly lead to even more
interesting relations for Gaussian polynomials.

\section*{Acknowledgement}

This research is done under the project "Application of graph theory" that is financed by BiH Minsitry of Science.

\end{document}